\documentclass[12pt]{amsart}

\usepackage{color}

\usepackage{amsmath}
\usepackage{amsthm}
\usepackage{amssymb}
\usepackage{amscd}
\usepackage{amsfonts}
\usepackage{amsbsy}
\usepackage[noadjust]{cite} 
\usepackage{epsfig,afterpage}
\usepackage{paralist}
\usepackage{psfrag}
\usepackage{mathrsfs}       
\usepackage{verbatim}

\newtheorem {theorem} {Theorem}

\newtheorem {proposition} [theorem]{Proposition}

\newtheorem {remark} [theorem]{\sc Remark}

\title[The Poincar\'e center problem]{A new approach to the Poincar\'e center problem}

\thanks{The author is partially supported by a MICIN grant number PID2020-113758GB-I00 and an AGAUR grant number 2021SGR-01618.}

\author[Isaac A. Garc\'ia and Jaume Gin\'e]{Isaac A. Garc\'ia and Jaume Gin\'e}

\address{Departament de Matem\`atica, Universitat de Lleida,
Avda. Jaume II, 69, 25001 Lleida, Spain}

\email{isaac.garcia@udl.cat}
\email{jaume.gine@udl.cat}



\begin{document}

\begin{abstract}
We address the classical (degenerate or non-degenerate) center problem posed by Poincar\'e in the 19th century for monodromic singularities of analytic families of planar vector fields $\mathcal{X}$. We prove that, generically, every analytic center admits a Laurent inverse integrating factor $V$ in weighted polar coordinates. Moreover, we show that when $\mathcal{X}$ has no local curves of zero angular speed, the Poincar\'e map is analytic. Based on this result, we derive a theoretical procedure to determine parameter constraints within the family that characterize centers without curves of zero angular speed. Applications to nontrivial families that have resisted other methods are also provided.
\end{abstract}

\maketitle

\section{Introduction and main results}

In this work, we focus on families of planar real analytic vector fields
$\mathcal{X} = P(x,y; \lambda) \partial_x + Q(x,y; \lambda) \partial_y$
defined in a neighborhood of a monodromic singularity, which can be placed at the origin without loss of generality. Here $\lambda \in \mathbb{R}^n$ denotes the finite set of parameters of the family.

Recall that a monodromic singular point of $\mathcal{X}$ is a singularity for which the associated flow rotates around it; consequently, a Poincar\'e map $\Pi$ is well defined on a sufficiently small transversal section with endpoint at the singularity on the origin. Centers and foci are examples of monodromic singularities. We restrict our attention to the monodromic parameter space $\Lambda \subset \mathbb{R}^n$, defined as the subset of parameters for which the origin is a monodromic singularity of $\mathcal{X}$.

Several algorithmic procedures are available to determine the parameter restrictions that define $\Lambda$; see, for example, \cite{Dum, Ar, AGR,AGR2, GGGrau}. The center problem consists in determining the subsets of $\Lambda$ corresponding to those $\mathcal{X}$ such that the  monodromic equilibrium point becomes a center, meaning that all nearby trajectories are closed orbits surrounding the equilibrium and therefore $\Pi$ is the identity map. A key simplification occurs when the singular point has no characteristic directions (see its definition below), hence $\Pi$ is analytic at the origin. Even in this simple case the center problem can be algebraically  unsolvable, see \cite{ADGG,BRY,CYZ,Ga-Ma,Ilyashenko-2,G1,GM,LWX,RS,TLZ} and the computation of a sufficiently large jet of $\Pi$ may be intractable. The harder center problem with characteristic directions is still open and there the function inverse integrating factor has played a relevant role, see \cite{AGG,GaGiGr,Ga-Gi,GaGi2,GS}. A first step for the solution of the center problem consists in to determine the stability of the singularity at first order that corresponds to the computation of the linear part of $\Pi$, also studied in several works, see for instance \cite{GaGi5,Ma,Me-2,Me-3}.

We consider the Newton diagram $\mathbf{N}(\mathcal{X})$ of $\mathcal{X}$; see \cite{Bruno, AGR}. Each edge of $\mathbf{N}(\mathcal{X})$ has slope $-p/q$, where $(p,q) \in \mathbb{N}^2$ are coprime. From now on, we denote by $W(\mathbf{N}(\mathcal{X})) \subset \mathbb{N}^2$ the set of all such weights $(p,q)$.

Introducing the {\it weighted polar} coordinates $(x,y) \mapsto (\varphi, \rho)$ given by $x= \rho^p \cos\varphi$, $y = \rho^q \sin\varphi$, and removing a suitable power of $\rho$ appearing as a common factor, $\mathcal{X}$ is transformed into the {\it polar vector field} $\mathcal{Z} = \Theta(\varphi, \rho) \partial_\varphi + R(\varphi, \rho) \partial_\rho$, with
$R(\varphi, 0)=0$ and $\Theta(\varphi, 0) \geq 0$, assuming without loss of generality that the local flow has counterclockwise orientation.

The vector field $\mathcal{Z}$ is defined on the cylinder $C$ given by
\begin{equation}\label{cilinder}
C =  \{ (\varphi, \rho) \in \mathbb{S}^1 \times \mathbb{R} :  0 \leq \rho \ll 1 \},
\end{equation}
where $\mathbb{S}^1 = \mathbb{R}/ (2 \pi \mathbb{Z})$. The set of {\it characteristic directions} at the origin of $\mathcal{X}$ is defined as
$\Omega_{pq} = \{ \varphi^* \in \mathbb{S}^1 : \Theta(\varphi^*, 0) = 0 \}$. Thus, the set $\{\rho=0 \}$ is invariant under the flow of $\mathcal{Z}$ and becomes either a periodic orbit or a polycycle, depending on whether $\Omega_{pq} = \emptyset$ or $\Omega_{pq} \neq \emptyset$, respectively. This is just the reason why $\Pi$ is analytic when $\Omega_{pq} = \emptyset$.
\newline

An inverse integrating factor $V(\varphi, \rho)$ of the polar vector field $\mathcal{Z}$ on $C$ is a real-valued function of class $C^1(C \setminus \{\rho=0 \})$ that satisfies the linear partial differential equation
\begin{equation}\label{edp-V}
\mathcal{Z}(V) = V \; {\rm div}(\mathcal{Z}),
\end{equation}
where ${\rm div}(\mathcal{Z}) = \partial_\varphi \Theta + \partial_\rho R$ denotes the divergence of $\mathcal{Z}$. Equivalently, \eqref{edp-V} can be written as the divergence-free condition ${\rm div}(\mathcal{Z}/V) \equiv 0$ off the zero-set $V^{-1}(0)$.

We say that $\mathcal{Z}$ possesses a {\it Laurent inverse integrating factor} in $C$ if it admits a Laurent series expansion about $\rho=0$ of the form
\begin{equation}\label{V-Laurent}
V(\varphi, \rho) = \sum_{j \in \mathbb{Z}} v_j(\varphi) \rho^{j},
\end{equation}
whose coefficients are bounded functions $v_j : \mathbb{S}^1 \to \mathbb{R}$.

Therefore, for any fixed $\bar{\varphi} \in \mathbb{S}^1$, the point $\rho=0$ for the function $V(\bar{\varphi}, \rho)$ is either regular at $\rho=0$, or it has a pole (both having leading term) or an essential singularity. In the particular case that \eqref{V-Laurent} becomes $V(\varphi, \rho) = v_m(\varphi) \rho^{m} + \cdots$ with $v_m(\varphi) \not\equiv 0$ and $m \in \mathbb{Z}$, we call $v_m$ the leading coefficient and $m$ the leading exponent.

We point out that only Laurent inverse integrating factors with leading term have been used in previous works like \cite{GaGi2, GaGi4, Ga-Ma}. In \cite{GaGi4}, it is proved that if the rescaled vector field $\mathcal{Z}/\Theta$ admits a Laurent inverse integrating factor (allowing the coefficients $v_j$ to be unbounded on $\Omega_{pq}$) without an essential singularity at $\rho=0$, then the origin of the restricted vector field $\mathcal{X}|_{\Lambda \setminus \Lambda_{pq}}$ is either a center or a focus of maximal order, where
\begin{equation}\label{L-pq}
\Lambda_{pq} = \{ \lambda \in \Lambda : \Theta^{-1}(0) \backslash \{ \rho = 0 \} \neq \emptyset \}.
\end{equation}
Here, we use the term {\it focus of maximal order} to denote a focus whose associated Poincar\'e map $\Pi$ is close to the identity and differs from it only by a single nonvanishing Poincar\'e--Lyapunov quantity; see Remark \ref{remark-principal}(i) for details.

Observe that $\mathcal{X}$ has no local curves of {\it zero angular speed} when restricted to $\Lambda \setminus \Lambda_{pq}$, although $\{\rho=0 \}$ may still be a polycycle. We refer the reader to the appendix of \cite{GaGi4} for the computation of $\Lambda_{pq}$, using either branching theory based on the Newton--Puiseux algorithm or techniques from the singularity theory of maps.

In Proposition 5 of \cite{GGR}, an example is given of a vector field $\mathcal{X}$ with a maximal-order nilpotent focus at the origin, for which $\Omega_{pq} = \emptyset$ (hence $\Lambda_{pq} = \emptyset$), and which does not admit a Laurent inverse integrating factor with leading term.

We show below that centers in the class $\Omega_{pq} = \emptyset$ always admit an analytic inverse integrating factor at $\rho=0$, that is, of the form \eqref{V-Laurent} with non-negative leading exponent.

\begin{proposition}\label{prop-monpq}
Any analytic center with $\Omega_{pq} = \emptyset$ admits an analytic inverse integrating factor $V(\varphi, \rho)$ at $\rho=0$.
\end{proposition}

\begin{remark}
{\rm In contrast to what occurs in the coordinates $(\varphi, \rho)$, the analyticity of inverse integrating factors in Cartesian coordinates $(x,y)$ is not guaranteed for centers with $\Omega_{pq} = \emptyset$. Let $\omega = P(x,y; \lambda) , dy - Q(x,y; \lambda) , dx$ be the differential 1-form associated with $\mathcal{X}$. We say that a $C^1$ real-valued function $v(x,y)$ is an inverse integrating factor of $\mathcal{X}$ if $\omega / v(x,y)$ is closed in a neighborhood of the origin, except on the zero set $v^{-1}(0)$. We note that there exist polynomial vector fields $\mathcal{X}$ with a center at the origin and $\Omega_{pq} = \emptyset$ that do not admit any inverse integrating factor $v(x,y)$ analytic at the origin. The reason for that is because sometimes any smooth $v(x,y)$ is necessarily flat at the origin; see the proof of Theorem 2 in \cite{GP}. }
\end{remark}

Below we show that, in general, inverse integrating factors $V(\varphi, \rho)$ are singular at $\rho=0$ once the condition $\Omega_{pq} = \emptyset$ is removed, as stated in Proposition \ref{prop-monpq}. The main results of this work are the following. First we need to consider the real-complex system
\begin{equation}\label{complex-eq}
\Theta(\varphi, \rho) = R(\varphi, z) = 0,
\end{equation}
with  $z \in \mathbb{C}$ and ${\rm Re}(z) = \rho$.

\begin{theorem}\label{center-V}
Any analytic center admits a Laurent inverse integrating factor $V$ of the form \eqref{V-Laurent} provided that \eqref{complex-eq} has isolated solutions $(\varphi, z) = (\varphi^*, 0)$ for any $\varphi^* \in \Omega_{pq}$. Moreover, this condition is guarantee when we restrict the parameters to $\Lambda \setminus \Lambda_{pq}$.
\end{theorem}

\begin{remark}
{\rm In the proof of Theorem \ref{center-V} we show that generically \eqref{complex-eq} has isolated solutions $(\varphi, z) = (\varphi^*, 0)$ for any $\varphi^* \in \Omega_{pq}$. Indeed we think it is always true in the monodromic context. }
\end{remark}

The next two remarks show examples of both centers and foci having Laurent inverse integrating factors with essential singularity at $\rho=0$.

\begin{remark}
{\rm The following example is extracted from \cite{BerroneGiacomini}. The vector field $\mathcal{X} = (-y + x (x^2+y^2)) \partial_x + (x + y (x^2+y^2)) \partial_y$ has a focus at the origin, and its associated polar vector field $\mathcal{Z} = \partial_{\varphi} + \rho^3 \partial_{\rho}$ admits the inverse integrating factor
\begin{equation}\label{V-focus-esencial}
V_{1}(\varphi, \rho) = \rho^3 \, \sin(2 \varphi + \rho^{-2})
\end{equation}
which is analytic in $C \setminus \{\rho = 0\}$. We observe that $V_1$ has the Laurent expansion at $\rho = 0$ without leading term given by
$$
V_{1}(\varphi, \rho) = \rho^3 \left( \sin(2\varphi) \cos(\rho^{-2})+ \cos(2\varphi) \sin(\rho^{-2}) \right)
$$
obtained after the Laurent expansion at $\rho=0$ of $\cos(\rho^{-2})$ and $\sin(\rho^{-2})$. It is worth noting that $\mathcal{Z}$ also has the inverse integrating factor $V_2(\varphi, \rho) = \rho^3$. }
\end{remark}

\begin{remark}
{\rm The vector field $\mathcal{X} = -y \partial_x + x \partial_y$ has a center at the origin, and its associated polar vector field $\mathcal{Z} = \partial_{\varphi}$ admits the inverse integrating factor $V_1(\varphi, \rho) = \sin(\rho^{-1}) + \sin(\rho)$, which is analytic in $C \setminus \{\rho = 0\}$ and exhibits an essential singularity at $\rho = 0$. We note that $\mathcal{Z}$ also possesses the inverse integrating factor $V_2(\varphi, \rho) = \rho$. }
\end{remark}

An important consequence of the technique used in the proof of Theorem \ref{center-V} concerns the structure of the Poincar\'e map. Following Il'yashenko's work \cite{Ilyashenko}, it follows that $\Pi$ may even be non-differentiable at the origin. Specifically, $\Pi$ is a semiregular map with a linear term, meaning it can be expressed as $\Pi(x) = \eta_1 x + o(x)$, where the linear leading coefficient satisfies $\eta_1 > 0$. Using the analytic dependence of $\mathcal{X}$ on $\lambda$, Medvedeva in \cite{Me} shows that $\Pi$ admits a Dulac asymptotic expansion of the form
\begin{equation}\label{Poinc-expan}
\Pi(x) = \eta_1 x + \sum_j \mathcal{P}_j(\log x) \, x^{\nu_j},
\end{equation}
where the exponents $\nu_j > 1$ are independent of $\lambda$ and tend to infinity, while $\mathcal{P}_j$ are polynomials whose coefficients depend analytically on the coefficients of $\mathcal{X}$.

It is well known that \eqref{Poinc-expan} contains no logarithmic terms and converges in a neighborhood of the origin (hence $\Pi$ is analytic at the origin) when $\Omega_{pq} = \emptyset$, since in that case $\rho = 0$ becomes a periodic orbit of $\mathcal{Z}$. The following result generalizes the previous one, extending it to the case where the parameters lie in $\Lambda \setminus \Lambda_{pq}$.

\begin{theorem}\label{center-Pi}
The monodromic singular point at the origin of any analytic planar vector field restricted to $\Lambda \setminus \Lambda_{pq}$ has associated an analytic Poincar\'e map $\Pi$ at the origin.
\end{theorem}

\begin{remark}
{\rm We want to emphasize that the existence or not of curves of zero angular speed may depend on the weights $(p,q) \in W(\mathbf{N}(\mathcal{X}))$. We show this behaviour with the following family
\begin{equation}\label{ejemplo3-JDE-Armengol}
\dot{x} = b x^2 y + a x y^2 - b y^3 - x^4, \ \ \ \dot{y} = 4 b x y^2 + a y^3 + 2 x^5.
\end{equation}
In the work \cite{Ga-Ma-Ma} it is proved that the origin of \eqref{ejemplo3-JDE-Armengol} has a monodromic singularity in the parameter set $\Lambda = \{ (a, b) \in \mathbb{R}^2 : b > 1/6 \}$ and moreover that $\eta_1 = \exp\left( \frac{2 \pi \sqrt{3}}{3} \frac{a}{b} \right)$ using polar coordinates. Therefore, the origin is an attractor (resp. repeller) focus if $a < 0$ (resp. $a > 0$) while it is a time-reversible center for $a=0$. We notice that in \eqref{ejemplo3-JDE-Armengol}  one has $W(\mathbf{N}(\mathcal{X})) = \{ (1,1), (1,2) \}$. First we see that $\Lambda_{11} = \emptyset$ because, using polar coordinates $(\varphi, \rho)$, the equation $\Theta(\varphi, \rho)$ is polynomial of degree 2 in $\rho$ with discriminant negative for any $\varphi$ in a neighborhood of $\Omega_{11} = \{0, \pi\}$ and $b \in \Lambda$. On the other hand, if now $(\varphi, \rho)$ denote (1,2)- weighted polar coordinates, then $\Theta(\varphi, \rho)$ is again a quadratic polynomial in $\rho$ but its discriminant is positive when $\varphi$ lies in a neighborhood of $\Omega_{12} = \{\pi/2, 3 \pi/2\}$ and $b \in \Lambda$ provided $|a| > 4 b$. In summary we have seen that $\Lambda_{12} \neq \emptyset$. }
\end{remark}

Restricting parameters to $\Lambda \setminus \Lambda_{pq}$ and assuming the existence of a Laurent inverse integrating factor for $\mathcal{Z}$ with leading term, it was shown in \cite{GaGi4} that $\Pi$ admits a formal expansion at the origin. In light of Theorem \ref{center-Pi}, we now know that $\Pi$ is in fact analytic at the origin. This represents a significant improvement over some results presented in Theorem 10 of \cite{GaGi4}, where additional (and unnecessary) conditions were imposed to establish the analyticity of $\Pi$.

On the other hand, we wish to point out that even though $\Pi$ is analytic at the origin, and therefore its expansion
\begin{equation}\label{Poinc-analytic}
\Pi(\rho_0) = \sum_{j\geq 1} \eta_j \rho_0^j
\end{equation}
is simpler than \eqref{Poinc-expan}, the computation of the coefficients $\eta_j$ (generalized Poincar\'e--Lyapunov quantities) remains an open problem. The difficulty arises from the fact that $\Pi(\rho_0) = \Phi(2 \pi; \rho_0)$, where the solution $\Phi(\varphi; \rho_0)$ of the Cauchy problem given by $d \rho / d \varphi = R(\varphi, \rho)/\Theta(\varphi, \rho)$ with initial condition $\Phi(0; \rho_0) = \rho_0 > 0$ sufficiently small is not analytic with respect to $\rho_0$ at $\rho_0 = 0$. An obvious exceptional case where Bautin's method can be applied to compute $\eta_j$ occurs when $\Omega_{pq} = \emptyset$, in which case $\Phi(\varphi; \rho_0) = \sum_{j\geq 1} c_j(\varphi) \rho_0^j$ and $\eta_j = c_j(2 \pi)$. Another (nontrivial) exceptional case where the $\eta_j$ can be computed is given by those monodromic singularities in $\Lambda \setminus \Lambda_{pq}$ that possess an explicit Laurent inverse integrating factor with leading term, as established in \cite{GaGi4} (see Remark \ref{remark-principal}).

\section{A procedure to solve some center-focus problems on $\Lambda \setminus \Lambda_{pq}$}

The key point of the method relies in the following result that indicates how to compute the coefficients $v_j(\varphi)$ with $j \geq 1$ of the Laurent expression  \eqref{V-Laurent} of $V(\varphi, \rho)$ when the initial condition $V(0, \rho)$ is analytic at $\rho=0$.

\begin{theorem} \label{teo-Pi-Algorithm}
We restrict to the parameter space $\Lambda \setminus \Lambda_{pq}$ and, by a rotation, we take $0 \not\in \Omega_{pq}$. Then the Laurent expansion \eqref{V-Laurent} with initial condition $V(0, \rho)$ analytic at $\rho=0$ becomes a convergent Taylor expansion
\begin{equation}\label{Taylor-V}
V(\varphi, \rho) \Big|_{I_{pq}} = \sum_{j \ge m} v_j(\varphi) \, \rho^j,
\end{equation}
with $m \geq 1$ and $I_{pq} = [0, 2 \pi] \setminus \Omega_{pq}$.
\end{theorem}

The condition $V(0, \rho)$ analytic at $\rho=0$ in Theorem \ref{teo-Pi-Algorithm} is very important. For example, Theorem \ref{teo-Pi-Algorithm} cannot be applied to the Laurent inverse integrating factor \eqref{V-focus-esencial} since it has an essential singularity at $\rho=0$ for any $\varphi$.

Based on Theorems \ref{center-V} and \ref{teo-Pi-Algorithm}, we propose a method to derive necessary center conditions for analytic families of vector fields $\mathcal{X}$ exhibiting a monodromic singularity on $\Lambda \setminus \Lambda_{pq}$. To this end, we use either polar coordinates or $(p,q)$-weighted polar coordinates $(\varphi, \rho)$ with $(p,q) \in W(\mathbf{N}(\mathcal{X}))$.

The first step consists in proposing the initial terms of a convergent Taylor expansion \eqref{Taylor-V} of $V(\varphi, \rho)$ valid in $[0, 2\pi] \setminus \Omega_{pq}$ where $m \in \mathbb{N}$ is arbitrary. By requiring that this series satisfy equation \eqref{edp-V} in $[0, 2\pi] \setminus \Omega_{pq}$, we obtain constraints on the parameters $\lambda$ of the family $\mathcal{X}$, ensuring that an initial string of coefficients $v_j$ with $j \geq m \geq 1$ is formed by well-defined functions on $\mathbb{S}^1$, in particular the coefficients $v_j$ must be bounded at $\Omega_{pq}$ and $2 \pi$-periodic. These constrains are center conditions.

We continue by computing further coefficients $v_j$ until either a saturation occurs, meaning no additional parameter constraints arise. In the former situation, if we are able to obtain a closed-form expression for $V$, rather than merely the first terms of its expansion \eqref{Taylor-V} in $[0, 2\pi] \setminus \Omega_{pq}$, and $V$ has not essential singularity at $\rho=0$ then we can apply the technique developed in \cite{GaGi4} to derive a set of necessary and sufficient center conditions (see Remark \ref{remark-principal} for details).

\begin{remark}\label{remark-principal}
{\rm Once we have determined a sufficiently long initial string of coefficients $v_j$ that are well-defined on $\mathbb{S}^1$ in the expansion \eqref{Taylor-V}, and we observe a possible stabilization in the sense that no further parameter conditions arise, then we can proceed as follows. Assuming $V$ exists and it has not essential singularity at $\rho=0$, that is, $V(\varphi, \rho) = \sum_{j \geq m} v_j(\varphi) \rho^{j}$ with $v_m(0)=1$, according to the results in \cite{GaGi4}, only two possibilities remain:

\begin{itemize}
  \item[(i)] The singularity becomes a focus of maximal order, admitting a Laurent inverse integrating factor with leading exponent $m$. In this case, the Poincar\'e map has a formal power series expansion of the form $\Pi(\rho) = \eta_1 \rho + \cdots$ when $m = 1$, or $\Pi(\rho) = \rho + \eta_m \rho^m + \cdots$ when $m \geq 2$. Moreover, $\Pi(\rho) = \rho$ holds only if $\eta_1 = 1$ or $\eta_m = 0$, respectively.

  \item[(ii)] The singularity becomes a center.
\end{itemize}

To distinguish between cases (i) and (ii) within the restricted monodromic parameter space $\Lambda \setminus \Lambda_{pq}$, it is necessary to obtain the closed-form expression of $V$, rather than only the first terms of its expansion near $\rho = 0$. In this context, setting $\mathcal{F} = R / \Theta$, we obtain that $\log \eta_1 = \mathfrak{g}$ when $m = 1$, and $\eta_m = \mathfrak{g}$ when $m > 1$, where $\mathfrak{g} = G(r)$ with
\begin{equation}\label{Function-G}
G(r) = \int_0^{2 \pi} \frac{\mathcal{F}(\varphi, r)}{V(\varphi, r)} \, d \varphi,
\end{equation}
and this expression is independent of $r$ for all sufficiently small $r > 0$.}
\end{remark}

We note that there exist maximal-order foci that do not admit a Laurent inverse integrating factor with leading term. This behavior is illustrated by the system $\dot{x}= y$, $\dot{y}= -x^5 + a x^4 y$, which has a maximal-order focus at the origin for $a \neq 0$, as can be verified by applying the Bautin method, given that $\Omega_{13} = \emptyset$. In Proposition 5 of \cite{GGR}, it is shown that this family of foci does not possess a Laurent inverse integrating factor with leading term because we reach some coefficient $v_j$ with $j > m$ not defined in $\mathbb{S}^1$.

\begin{remark}\label{remark-cilindro}
{\rm In order to apply the results from \cite{GGR}, instead of seeking a Laurent inverse integrating factor $V$ for $\mathcal{Z}$ with leading term, we may equivalently look for a Laurent inverse integrating factor $V^*(\varphi, \rho)$ for the vector field $\mathcal{Z}^* = \mathcal{Z}/\Theta = \partial_\varphi + \mathcal{F}(\varphi, \rho) \partial_\rho$ defined on $C \setminus \Theta^{-1}(0)$. The relationship between $V$ and $V^*$ is simply $V^* = V / \Theta$. In this setting, we observe that the coefficients $v_j$ in the Laurent expansions of $V^*$ near $\rho = 0$ are functions defined only on $\mathbb{S}^1 \setminus \Omega_{pq}$ and may become unbounded near $\Omega_{pq}$. }
\end{remark}

\subsection{Computation of the coefficients $v_j$ of the Taylor expansion \eqref{Taylor-V}} \label{SS-ascending}

For any $\varphi \not\in \Omega_{pq}$, we put the expansion $\sum_{j \geq m} v_j(\varphi) \rho^{j}$ with arbitrary $m \in \mathbb{N}$ for the $V$ in the partial differential equation \eqref{edp-V} together with $\Theta(\varphi, \rho) = \sum_{j \geq 0} G_j(\varphi) \rho^{j}$ and $R(\varphi, \rho) = \sum_{j \geq 1} R_j(\varphi) \rho^{j}$ yielding that the coefficients $v_j$ are solutions of $2 \pi$-periodic singular linear differential equation
\begin{equation}\label{edo-para-vj}
\mathcal{L}^{a}_j(v_j) + Q_j = 0
\end{equation}
where $\mathcal{L}^{a}_{m+j}(v_{m+j}) := G_0 v'_{m+j} + ((m+j-1) R_1 - G'_0) v_{m+j}$ for $j \geq 0$, with $Q_m \equiv 0$ and where $Q_j(\varphi)$ with $j > m$ is a function depending linearly on the $v_k$ and $v'_k$ with $k < j$. In particular, the leading coefficient $v_m$ is a nontrivial solution of the homogeneous linear differential equation
\begin{equation}\label{EDO-GM}
G_0 v'_m + ((m-1) R_1 - G'_0) v_m = 0.
\end{equation}
If we fix $m=1$ then \eqref{EDO-GM} has all its nontrivial solutions $v_m(\varphi; C) = C \, G_0(\varphi)$ analytic on $[0, 2\pi]$ with constant $C \neq 0$. Otherwise, except special cases, only analytic solutions with $m > 1$ of \eqref{EDO-GM} are guarantee in each connected open component $I_k$ of $[0, 2\pi] \setminus \Omega_{pq}$ since \eqref{EDO-GM} is singular at $\Omega_{pq}$. More specifically, when $m > 1$, on each component we have
\begin{equation}\label{vm-explicit}
v_m(\varphi; C_m) = C_m G_0(\varphi) \exp\left( (1-m) \, \mathcal{P}(\varphi) \right)
\end{equation}
with arbitrary $C_m \neq 0$ and $\mathcal{P}' = R_1/G_0$. In order to $v_m$ with $m >1$ will be defined at $\Omega_{pq}$ (excluding even removable discontinuities of $v_m$ at $\Omega_{pq}$) two possibilities arise regarding the behaviour of $\mathcal{P}$ at $\Omega_{pq}$:

\begin{itemize}
  \item[(i)] Either $\mathcal{P}|_{\Omega_{pq}}$ is bounded (sometimes we will need to take a continuous primitive $\mathcal{P}$ at $\Omega_{pq}$ that will be piecewise defined on $[0, 2 \pi]$) or it is not at some $\varphi^* \in \Omega_{pq}$ but then $\lim_{\varphi \to \varphi^*} \mathcal{P}(\varphi) = + \infty$ and $v_m(\varphi^*; C_m) =0$. In any case we get that the function \eqref{vm-explicit} is well defined on $[0, 2 \pi]$.

  \item[(ii)] Otherwise \eqref{vm-explicit} has no solution on $[0, 2 \pi]$ except for $m=1$.
\end{itemize}

Finally, we will impose that a solution of \eqref{vm-explicit} on $[0, 2 \pi]$ be defined in $\mathbb{S}^1$ imposing its $2\pi$-periodicity. This is guarantee for $m=1$. When $m > 1$ we fall in the former case (i) then this condition is merely $\mathcal{P}(2 \pi) = \mathcal{P}(0)$. Notice that when additionally $\mathcal{P}$ is continuous at $\Omega_{pq}$ then the Cauchy principal value
\begin{equation}\label{xipq}
\xi_{pq} = PV \int_{0}^{2 \pi} \frac{R_1(\varphi)}{G_0(\varphi)} d \varphi,
\end{equation}
used in some works, exists and takes the value $\xi_{pq} =  \mathcal{P}(2 \pi) - \mathcal{P}(0)$. Therefore $\xi_{pq} = 0$ is a necessary center condition when $m >1$ and $\mathcal{P}$ is continuous at $\Omega_{pq}$. This result is in agreement with the equation $(m-1) \xi_{pq} = 0$ appearing in Proposition 2 of \cite{GGR} which is only valid for Laurent inverse integrating factors with leading exponent $m$ and under some extra technical assumption.

\subsection{How the method works in a toy example}

We consider the linear vector field $\mathcal{X} = (-y + \lambda x) \partial_x + (x + \lambda y) \partial_y$ and we are going to see that the origin is a center if and only if $\lambda=0$.

We have $W(\mathbf{N}(\mathcal{X})) = \{ ((1, 1) \}$ and taking polar coordinates $(\varphi, \rho)$ its associated polar vector field is $\mathcal{Z} = \partial_\varphi + \lambda \rho \partial_\rho$ so that $\Omega_{11} = \emptyset$ and therefore $\Lambda_{11} = \emptyset$ and $\Lambda = \mathbb{R}$. This implies that the differential equations $\mathcal{L}^{a}_j(v_j) = 0$ are non-singular and we only have to compute the coefficients $v_j$ defined in $\mathbb{S}^1$ if they exist. Straightforward computations yield $\mathcal{L}^{a}_m(v_m) = v'_m + (m-1) \lambda v_m = 0$, hence $v_m(\varphi; C_m)= C_m \exp((1-m) \lambda \varphi)$ with $C_m \neq 0$. The $2 \pi$-periodicity of $v_m$ gives the center condition $(m-1) \lambda = 0$.

We assume that $\lambda \neq 0$ and we take $m=1$. Then $v_1(\varphi)=1$ and the only $2 \pi$-periodic solution of each equation $\mathcal{L}^{a}_j(v_j) = 0$ is the trivial one $v_j(\varphi)= 0$ for all $j >1$. At this point we see that $V(\varphi, \rho) = \rho$ is an inverse integrating factor of $\mathcal{Z}$ and therefore the origin is either a focus or a maximal order focus. To discern between these two possibilities we compute, according to \eqref{Function-G},
$$
G(r) = \int_0^{2 \pi} \frac{\mathcal{F}(\varphi, r)}{V(\varphi, r)} d \varphi = \int_0^{2 \pi} \lambda d \varphi = 2\pi \lambda.
$$
Since $m=1$, we have $\eta_1 = \exp(2 \pi \lambda) \neq 1$ so the origin is a focus.

\section{Non-trivial examples}

\subsection{Ma\~{n}osa monodromic family I}

In the work \cite{Ma} it is analyzed the 1-parameter family
\begin{equation}\label{ejemplo-manyosa}
\dot{x} = x y^2 - y^3 + a x^5, \ \ \dot{y} = 2 x^7 - x^4 y + 4 x y^2 + y^3,
\end{equation}
that has a monodromic singular point at the origin when it is restricted to the parameter set $\Lambda = \{a \in \mathbb{R} : \Delta(a) := 32 - (1+3 a)^2 >0 \}$. In \cite{GaGi3} this family was named Ma\~{n}osa monodromic family I when the sufficient focus condition developed there was applied to show that the origin of \eqref{ejemplo-manyosa} is a focus in $\{ a >0 \} \cap \Lambda$. Indeed, in \cite{Ma} the author proves a strong result showing that on $\Lambda$ the linear coefficient of $\Pi$ is
\begin{equation}\label{eta1-manyosa}
\eta_1 = \exp\left( \pi + \frac{4 \pi a}{\sqrt{\Delta(a)}} \right).
\end{equation}
Nevertheless the center or focus nature of the origin of \eqref{ejemplo-manyosa} when $a = -31/25$ (corresponding to $\eta_1 = 1$) is unknown as far as we know, so we will apply our method to solve this problem.

\begin{proposition}\label{prop-manyosa-I-foco}
The origin of family \eqref{ejemplo-manyosa} restricted to $\Lambda$ is a focus.
\end{proposition}

\begin{proof}
As we mentioned before, using the result of \cite{Ma}, we only need to prove that system \eqref{ejemplo-manyosa} with $a = -31/25$ is a focus. Family \eqref{ejemplo-manyosa} has $W(\mathbf{N}(\mathcal{X})) = \{ (1, 3), (1, 1) \}$. Taking polar coordinates $(\varphi, \rho)$ we check that $0 \in \Omega_{11}$. Therefore, we permute coordinates $(x,y) \mapsto (y,x)$ in \eqref{ejemplo-manyosa} and we reach, taking again polar coordinates in the transformed system, that $G_0(\varphi) = \cos^2\varphi (-5 + 3 \cos( 2\varphi))/2$, hence $\Omega_{11} = \{\pi/2, 3 \pi/2\}$.

First we prove that there are no curves of zero angular speed. We see that $\Theta(\varphi, \rho) = - \cos^4\varphi + (a+1) \rho^2 \cos\varphi \sin^5\varphi - 2 \rho^4 \sin^8\varphi - \sin^2(2\varphi) = 0$ is a biquadrate equation in $\rho$. A necessary condition to have real solutions $\rho=\rho^*(\varphi, a)$ in a neighborhood of $\Omega_{11}$ and $\Lambda$ is that the discriminant $\delta(\varphi, a) = -39 + a (2 + a) - (-23 + a (2 + a)) \cos(2\varphi)$ be positive in that region. From the continuity of $\delta$ and the fact that $\delta|_{\Omega_{11}} = -62 + 2 a (2 + a) < 0$ on $\Lambda$ we deduce that $\Lambda_{11} = \emptyset$.

We see that
$$
\frac{R_1(\varphi)}{G_0(\varphi)}= \frac{2 + 6 \cos\varphi \sin\varphi}{-5 + 3 \cos(2\varphi)}
$$
has a primitive $p(\varphi) = (-\arctan(2 \tan\varphi) - \log(5 - 3 \cos(2 \varphi))/2$ that is not continuous at $\Omega_{11}$ since $\arctan(2 \tan\varphi)$ has a finite jump discontinuities at $\Omega_{11}$ taking the principal branch of $\arctan$. For that reason we take the piecewise primitive
$$
\mathcal{P}(\varphi) = \begin{cases}
p(\varphi) & \varphi \in [0, \pi/2), \\
p(\varphi)-  \pi/2 & \varphi \in [\pi/2, 3\pi/2), \\
p(\varphi) - \pi & \varphi \in [3\pi/2, 2\pi],
\end{cases}
$$
that is continuous on $[0, 2 \pi]$. Indeed, $\mathcal{P}$ is analytic on $[0, 2 \pi]$ since $\mathcal{P}'(\varphi) = p'(\varphi) = R_1/G_0$ is analytic there. Then there are analytic coefficient functions $v_m(\varphi; C_m)$ for any $m > 1$ given by \eqref{vm-explicit} but none of them are $2 \pi$-periodic since $\mathcal{P}$ is not.

Therefore we need to take $m=1$ so that $v_1(\varphi; C_1) = C_1 \, G_0(\varphi)$ with $C_1 \neq 0$. The next coefficient satisfies the equation $\mathcal{L}^{a}_2(v_2) = G_0(\varphi) v'_2(\varphi) + P_2(\varphi) v_2(\varphi) = 0$ is again homogeneous with $P_2(\varphi) = \cos\varphi (4 \cos\varphi - 17 \sin\varphi + 15 \sin(3\varphi))/4$. It is easy to see, like we did for $v_m$ previously, that the only $2 \pi$-periodic solution of this equation is the trivial one $v_2(\varphi) \equiv 0$.

The differential equation \eqref{edo-para-vj} with $j=3$ that $v_3$ satisfies is no longer homogeneous since $Q_3 \not\equiv 0$. We let $v_3^h(\varphi)$ to be the solution of the associated homogeneous equation $\mathcal{L}^{a}_3(v_3^h) = 0$. Up to a multiplicative constant it has the form $v_3^h(\varphi) = \cos^2\varphi  (5 - 3 \cos(2 \varphi))^2  \exp(F_3(\varphi))$ where
$$
F_3(\varphi) = \begin{cases}
f_3(\varphi) & \varphi \in [0, \pi/2), \\
f_3(\varphi) + \pi & \varphi \in [\pi/2, 3\pi/2), \\
f_3(\varphi) + 2\pi & \varphi \in [3\pi/2, 2\pi],
\end{cases}
$$
with $f_3(\varphi) = \arctan(2 \tan\varphi)$ in order to have $v_3^h$ analytic on $[0, 2 \pi]$. We perform the change of dependent variable $v_3(\varphi; C_3) = v_3^h(\varphi) w_3(\varphi; C_3)$ and we obtain that
$$
w_3'(\varphi; C_3) = \frac{2 W_3(\varphi) \sin^2\varphi \tan^2\varphi}{(5 - 3 \cos(2 \varphi))^3} \exp(-F_3(\varphi))
$$
with $W_3(\varphi) = -7 - 17 a + (1 + 7 a) \cos(2 \varphi)$.

Taking $a = -31/25$, se see that $W_3 > 0$ in $[0, 2 \pi]$, hence $w_3' > 0$ there and, since $F_3$ is monotonous increasing we deduce that $v_3$ cannot be $2 \pi$-periodic. We conclude that there is no Laurent inverse integrating factor of \eqref{ejemplo-manyosa} with $a=-31/25$ finishing the proof.
\end{proof}

\subsection{The (3,5)-semi-homogeneous family}

We analyze the center-focus problem at the origin of the monodromic $(n, N)$-semi-homogeneous family $\mathcal{X} = P_n(x,y) \partial_x + Q_N(x,y) \partial_y$ where $P_n$ and $Q_N$ are homogeneous polynomials of odd degrees $n$ and $N$, respectively, and $n< N$. The associated differential system is
\begin{equation}\label{semi-mM}
\dot{x} = P_n(x,y) = -y^n + \sum_{\stackrel{i+j=n}{\scriptscriptstyle i \neq 0}} a_{ij} x^i y^j, \ \ \dot{y} = Q_N(x,y) = x^N + \sum_{\stackrel{i+j=N}{\scriptscriptstyle j \neq 0}} b_{ij} x^i y^j,
\end{equation}
restricted to the monodromic parameter space $\Lambda$.

We analyze the lower degree degenerate case $(3, 5)$-semi-homogeneous family. In this case it is tedious but straightforward to check using the monodromy criterium described in \cite{AGR,AGR2} that
\begin{equation}\label{Lambda-35}
\Lambda = \{  a_{30} = 0, a_{12}^2 + 4 a_{21} < 0 \}.
\end{equation}

\begin{proposition}\label{prop-seihomo-35}
Let the origin of the monodromic $(3, 5)$-semi-homogeneous family \eqref{semi-mM} be a center. Then $L_1 = 0$, where
\begin{eqnarray*}
L_1 &=& a_{12}^5 + 5 a_{12}^3 a_{21} + 5 a_{12} a_{21}^2 - 2 a_{21}^5 b_{05} + a_{12} a_{21}^4 b_{14} -  a_{12}^2 a_{21}^3 b_{23} \\
 &  & - 2 a_{21}^4 b_{23} + a_{12}^3 a_{21}^2 b_{32} + 3 a_{12} a_{21}^3 b_{32} -
 a_{12}^4 a_{21} b_{41} - 4 a_{12}^2 a_{21}^2 b_{41} -
 2 a_{21}^3 b_{41}.
\end{eqnarray*}
\end{proposition}

\begin{proof}
The $(3, 5)$-semi-homogeneous family \eqref{semi-mM} has $W(\mathbf{N}(\mathcal{X})) = \{ (1, 2), (1, 1) \}$. Taking polar coordinates we have that $0 \in \Omega_{11}$, hence we permute coordinates $(x,y) \mapsto (y,x)$ and, taking again polar coordinates in the transformed system, we get that $G_0(\varphi) = \cos^2\varphi (- \cos^2\varphi + a_{12} \sin\varphi \cos\varphi  + a_{21} \sin^2\varphi) \leq 0$ and $\Omega_{11} = \{\pi/2, 3 \pi/2\}$ when we restrict to $\Lambda$.

The equation $\Theta(\varphi, \rho) = G_0(\varphi) + G_2(\varphi) \rho^2 = 0$ is a quadratic equation in $\rho$. A necessary condition to have real positive solutions $\rho=\rho^*(\varphi, a, b) = \sqrt{-G_0/G_2}$ in a neighborhood of $\Omega_{11}$ and $\Lambda$ is that $G_2 > 0$ there. But $G_2$ is continuous and $G_2|_{\Omega_{11}} = -1 < 0$, hence $\Lambda_{11} = \emptyset$.

We have that $R_1(\varphi)/G_0(\varphi) = \tan\varphi$ with primitive $p(\varphi) = -\log(|\cos\varphi|)$ noncontinuous at $\Omega_{11}$ going to $+ \infty$. Then we cannot build any primitive bounded at $\Omega_{11}$. Then the function $v_m$ with $m>1$ given in \eqref{vm-explicit} is $v_m(\varphi; C_m) = C_m G_0(\varphi) \, |\cos^{m-1}\varphi|$ with $C_m \neq 0$.

Instead of analyzing the next coefficient $v_{m+1}$, we prefer to take $m=1$ that is justified using Theorem \ref{teo-Pi-Algorithm} with any initial analytic condition $V(0,\rho) = c_1 \rho + \cdots$ with $c_1 \neq 0$. Then
$v_1(\varphi; 1) = G_0(\varphi)$ with $C_1 =1$ without loss of generality.

The next coefficient satisfies an homogeneous equation $\mathcal{L}^{a}_2(v_2) = 0$ with general solution $v_2(\varphi; C_2) = C_2 G_0(\varphi) \, |\cos\varphi|$. We take $C_2=0$ (hence just $v_2\equiv 0$) choosing the compatible analytic initial condition $V(0,\rho) = c_1 \rho + O(\rho^3)$ so that $v_2(0; C_2) = C_2 = 0$.

The differential equation \eqref{edo-para-vj} with $j=3$ that $v_3$ satisfies is no longer homogeneous since $Q_3 \not\equiv 0$. We let $v_3^h(\varphi)$ to be the solution of the associated homogeneous equation $\mathcal{L}^{a}_3(v_3^h) = 0$. Up to a multiplicative constant it has the form $v_3^h(\varphi) = -2 \cos^2\varphi  G_0(\varphi)$. We express $v_3(\varphi; C_3) = v_3^h(\varphi) (w_3(\varphi)+ C_3)$ and we obtain that
$$
w_3(\varphi) = \frac{L_1}{\sqrt{-\Delta}} \, F_3(\varphi) + G_{3}(\varphi)
$$
being $\Delta = a_{12}^2 + 4 a_{21} < 0$ in $\Lambda$ and $F_3(\varphi)$ is a continuous function on $[0, 2 \pi]$ constructed piecewise where in each of the 3 components of $\Omega_{11} \setminus [0, 2 \pi]$ it is just to add some constant to the function
$$
{\rm arctan}\left( \frac{a_{12} -2 a_{21} \tan\varphi}{\sqrt{-\Delta}}\right).
$$
Moreover, $G_{3}$ is a function defined in $\mathbb{S}^1 \setminus \Omega_{11}$. However, $\lim_{\varphi \to \varphi^*} v_3^h(\varphi) G_{3}(\varphi)= -1/2$ when $\varphi^* \in \Omega_{11}$ and $v_3^h \, G_{3}$ can be extended by continuity to $\Omega_{11}$. Clearly this property is also shared by $v_3$. The necessary center condition determined by imposing $v_{3}$ be defined in $\mathbb{S}^1$ turns out to be that $L_1=0$ since $F_3$ is monotonous, hence not $2 \pi$-periodic.
\end{proof}

Below we present the complete center classification of the 1-parameter subfamily of the $(3, 5)$-semi-homogeneous vector field \eqref{semi-mM} given by
\begin{equation}\label{semi-35-particular}
\dot{x} = -y^3 - 2 x^2 y - 2 x y^2, \ \ \dot{y} = x^5 + x^4 y + a x^3 y^2 + \frac{1}{2} (2 a-1) x^2 y^3 - x y^4.
\end{equation}

\begin{proposition}\label{prop-seihomo-35-particular}
The origin of the $(3, 5)$-semi-homogeneous 1-parameter family \eqref{semi-35-particular} is monodromic for all $a \in \mathbb{R}$. Moreover, it is a center if and only if $3 + 2 a=0$.
\end{proposition}
\begin{proof}
In this family the quantity $L_1$ given in Proposition \ref{prop-seihomo-35} is $L_1 = 8 (3 + 2 a)$, hence $3 + 2 a = 0$ as a necessary center condition for the origin of \eqref{semi-35-particular}. Indeed it is also sufficient because, when $a = -3/2$, system \eqref{semi-35-particular} becomes
$$
\dot{x} = -y \, f(x,y), \ \ \dot{y} = \frac{1}{2} x (x^2 - 2 y^2) \, f(x,y),
$$
with $f(x,y)=2 x^2 + 2 x y + y^2$ and the origin is a center because $f$ has an isolated zero at $(x,y)=(0,0)$ so that the system is orbitally equivalent to a time-reversible vector field.
\end{proof}

\section{Proofs}

\subsection{Proof of Proposition \ref{prop-monpq}}

\begin{proof}
Let $\mathcal{Z} = \Theta(\varphi, \rho) \partial_\varphi + R(\varphi, \rho) \partial_\rho$ be its polar vector field and $\hat{\mathcal{Z}} = \mathcal{Z}/ \Theta = \partial_\varphi + \mathcal{F}(\varphi, \rho) \partial_\rho$ with $\mathcal{F} = R/\Theta$ be the associated vector field to $\mathcal{X}$ on the cylinder $C$ defined in \eqref{cilinder}. A parameterization of the Poincar\'e return map $\Pi$ is given by $\Pi(\rho_0) = \Phi(2 \pi; \rho_0)$ where $(\varphi, \Phi(\varphi; \rho_0))$ is the orbit of $\hat{\mathcal{Z}}$ trough the point $(\varphi, \rho) = (0, \rho_0)$ with $\rho_0 > 0$ sufficiently small.

Let the origin be a center of $\mathcal{X}$, that is we have $\Pi(\rho_0) = \rho_0$. We will use an embedding flow argument \cite{EncPer} to $\Pi$. Notice that $\hat{\mathcal{Z}}$ is analytic in $C$ because $\Omega_{pq} = \emptyset$. Since $\Pi$ is analytic and it is also the time--$2 \pi$ flow of the analytic vector field $\mathcal{Z}^\dag = \partial_\varphi$ also defined on $C$, it follows by Lemma 8 of \cite{Li-LLi-Xia} that $\hat{\mathcal{Z}}$ and $\mathcal{Z}^\dag$ are analytically equivalent. In particular it follows that there is an analytic diffeomorphism $\zeta$ defined on $C$ such that $\hat{\mathcal{Z}}= \zeta^* \mathcal{Z}^\dag$. The field $\mathcal{Z}^\dag$ possesses (among others) the inverse integrating factor $V^\dag(\varphi, \rho) = 1$. Denoting by $J$ the Jacobian determinant of $\zeta$ and going back to the original coordinates, we obtain that  $V = (V^\dag \circ \zeta)/J = 1/J$ is an analytic inverse integrating factor of $\hat{\mathcal{Z}}$ defined on $C$ because $J \neq 0$.
\end{proof}

\subsection{Proof of Theorem \ref{center-V}}

\begin{proof}
We define $C^+ = C \setminus \{\rho=0 \}$ and we let $\Psi(t; \rho_0)$ be the (analytic) flow associated to the polar vector field $\mathcal{Z}$ with initial condition $\Psi(0; \rho_0) = (0, \rho_0) \in C^+$, that is for $\rho_0>0$ and sufficiently small. We apply the characteristic method to solve the partial differential equation \eqref{edp-V} in $C^+$. We find that
\begin{equation}\label{charact-V}
v(t, \rho_0) := V \circ \Psi(t; \rho_0) = V(0, \rho_0) \, \exp\left( \int_0^t {\rm div}(\mathcal{Z}) \circ \Psi(s; \rho_0) \, ds \right),
\end{equation}
where $V(0, \rho_0)$ is an arbitrary initial condition that we take to get a Cauchy problem for the partial differential equation \eqref{edp-V}. Of course we will assume that $V(0, \rho_0) \not\equiv 0$ in order to avoid the trivial solution of \eqref{edp-V} that is not an inverse integrating factor by definition.

Notice that $(\varphi, \rho) = \Psi(t; \rho_0)$ is a bijection and defines an analytic change of variables since the orbits of $\mathcal{Z}$ foliate $C^+$, that is,
$$
C^+ = \bigcup_{\stackrel{0 < \rho_0 \ll 1}{\scriptscriptstyle 0 \leq t \leq T(\rho_0)}} \{  \Psi(t; \rho_0) \},
$$
where $T$ is the flight return time function $T : \Sigma^+ \to \mathbb{R}$ defined as $\Psi(T(\rho_0); \rho_0) = (2 \pi, \Pi(\rho_0))$. Since the initial condition $V(0, \rho_0)$ is arbitrary, we take it analytic in $0 < \rho_0 \ll 1$ so that, from \eqref{charact-V}, we deduce that the function $v$ becomes also analytic in $\{ 0 < \rho_0 \ll 1 \} \times \{ 0 \leq t \leq T(\rho _0) \}$. Once we have $v(t, \rho_0)$, clearly we construct $V(\varphi, \rho) = v(\Psi^{-1}(\varphi, \rho))$ that is also analytic in the strip $[0, 2 \pi] \times \{ 0 < \rho \ll 1 \}$. In order to ensure that $V$ is indeed analytic in $C^+$, that is, to ensure that $V$ is $2 \pi$-periodic function of $\varphi$ we will need to assume that the origin must be a center as we show below.

We parameterize with $\rho_0$ the transversal section $\Sigma^+ = \{ \varphi = 0 \} \subset C^+$ to the flow of $\mathcal{Z}$.  Evaluating \eqref{charact-V} at $t=T(\rho_0)$ gives
$$
V(2 \pi, \Pi(\rho_0)) = V(0, \rho_0) \, \exp\left( \int_0^{T(\rho_0)} {\rm div}(\mathcal{Z}) \circ \Psi(t; \rho_0) \, dt \right).
$$
If now we assume that the origin is a center of $\mathcal{X}$, that is $\Pi(\rho_0)=\rho_0$, we deduce that $\int_0^{T(\rho_0)} {\rm div}(\mathcal{Z}) \circ \Psi(t; \rho_0) \, dt = 0$ since the periodic orbits of the center are non-hyperbolic. This can be also deduced from the fundamental equation $\hat{V}(2 \pi, \Pi(\rho_0)) = \hat{V}(0, \rho_0) \Pi'(\rho_0)$ relating $\Pi$ and $\hat{V} = V/\Theta$, see \cite{GaGi4}. Therefore, in the center case the function $V$ satisfies $V(2 \pi, \rho_0) = V(0, \rho_0)$ identically and consequently $V$ is an analytic function well defined in the cylinder $C^+$.
\newline

In the following we will analyze the behavior of $V(\varphi, \rho)$ at $\rho=0$ using complex analysis. First we observe that the flow of $\mathcal{Z}$ can be extended to a neighborhood of $\mathbb{S}^1 \times \{\rho=0\}$ given by the extended cylinder $\hat{C}  \, = \, \left\{ (\varphi, \rho) \in \mathbb{S}^1 \times \mathbb{R} \, : \, 0 \leq |\rho| \ll 1 \right\}$. This extension allows us to use the characteristics  method in $\hat{C} \setminus \{ \rho=0 \}$ so that we can construct, in the center case, an inverse integrating factor $V$ that is analytic in the punctured cylinder $\hat{C} \setminus \{ \rho=0 \}$.

We take any fixed $\varphi \in  \mathbb{S}^1$, and we consider the function $F_{\varphi} : \hat{C} \setminus \{ \rho=0 \} \cap \{ \varphi = \bar{\varphi} \} \to \mathbb{R}$ defined by $F_{\varphi}(\rho) = V(\varphi, \rho)$. This function $F_{\varphi}$ is analytic in $I_0 = I \setminus \{0\}$ with $I \subset \mathbb{R}$ a small neighborhood of the origin and can exhibit an isolated singularity at $\rho=0$. In order to see that $F_{\varphi}$ has a Laurent series representation at $\rho=0$ and convergent in $I_0$ we need to show that there is a complex function $\hat{F}(z)$ with $z \in \mathbb{C}$ and $\rho = {\rm Re}(z)$ defined in a punctured disc $D_0 \subset \mathbb{C}$ of the origin in $\mathbb{C}$ with $I_0 \subset D_0$ such that the restriction of $\hat{F}(z)$ to $I$ is $F_{\varphi}$.

By standard arguments, see the proof of Theorem \ref{teo-Pi-Algorithm}, we know that if $V(0, \rho)$ is analytic at $\rho=0$ then $V(\varphi, \rho)$ is analytic in $[0, 2\pi] \times I$,  being $I$ a neighborhood of the origin, except at $(\varphi, \rho) = (\varphi^*, 0)$ with $\varphi^* \in \Omega_{pq}$. So we only need to show that the function $F_{\varphi^*}$ admits a Laurent series representation.

We define the complex function $R(\varphi, z)$ just by replacing, in the Taylor series centered at some point $(\varphi_0, \rho_0) \in \hat{C}$ of $R(\varphi, \rho)$, the real variable $\rho$ by the complex one $z$. Then, for each fixed $\varphi$, the complex function $R$ is holomorphic in a disc centered at $\rho_0$. We consider the real-complex differential system
\begin{equation}\label{complex-edo}
\begin{array}{lll}
\dot{\varphi} &=& \Theta(\varphi, \rho) = G_0(\varphi) + G_p(\varphi) \rho^p + O(\rho^{p+1}), \\
\dot{z} &=& R(\varphi, z) = R_q(\varphi) z^q + O(z^{q+1}),
\end{array}
\end{equation}
where $G_p \not\equiv 0$ and $R_q \not\equiv 0$ with $p\geq 1$ and $q \geq 1$. Clearly the line $z=0$ is invariant for \eqref{complex-edo} and the only singularities of \eqref{complex-edo} on $z=0$ are $(\varphi, z) = (\varphi^*, 0)$ with $\varphi^* \in \Omega_{pq}$. We are going to analyze when these singularities are isolated or not. We look for solutions of $\Theta(\varphi, \rho) = R(\varphi, z) = 0$ with $z \neq 0$. Taking into account that $R(\varphi, z) = z \,  \tilde{R}(\varphi, z)$ we consider the branches $z_j(\varphi)$ such that $\tilde{R}(\varphi, z_j(\varphi)) \equiv 0$ for any $\varphi$ in a neighborhood  (or semi-neighborhood) of $\varphi^*$ with $z_j(\varphi^*)=0$. Since by monodromy $R_1(\varphi^*) = 0$, we cannot apply the complex Implicit Function Theorem because $\partial_z \tilde{R}(\varphi^*, 0) = 0$. Anyway the number of such branches $z_j$ is finite by the complex Weierstrass Preparation Theorem and, moreover they are parametrizables by Newton-Puiseux Theorem as a convergent power series in the variable $(\varphi-\varphi^*)^{1/n}$ for some positive integer $n$, that is, $z_j(\varphi) = c_j \, (\varphi-\varphi^*)^{s/n} + \cdots$ with complex $c_j \neq 0$ and integer $s \geq 1$. We also have the following real Taylor expansions at $\varphi=\varphi^*$ that, by monodromy, are $G_0(\varphi) = \alpha (\varphi-\varphi^*)^{2k} + \cdots$ with $k \geq 1$, $\alpha \neq 0$ and $G_p(\varphi) = \beta  (\varphi-\varphi^*)^{m} + \cdots$ with $\beta \neq 0$ and $m \geq 0$.

We now consider the equation $\Theta(\varphi, \rho_j(\varphi)) = 0$ with $\rho_j(\varphi) = {\rm Re}(z_j(\varphi))$ and insert there the above expansions. The analysis is simpler recasting the expressions using the real variable $\varepsilon = (\varphi-\varphi^*)^{1/n}$. Using that $\rho_j(\varphi) = {\rm Re}(c_j) \, (\varphi-\varphi^*)^{s/n} + \cdots$, the outcome is that the real function $h_j(\varepsilon) := \Theta(\varphi^* + \varepsilon^n, \rho_j(\varphi^* + \varepsilon^n))$ has a convergent power series expansion $h_j(\varepsilon) = \alpha \varepsilon^{2 k n} + \beta \, {\rm Re}(c_j) \varepsilon^{m n+p s} + \cdots$ where here the dots denote higher order terms with respect to $\min\{2 k n, m n+p s \}$. Notice that $h_j \not\equiv 0$ is just the condition for having isolated solutions because $h_j$ would be a nonzero real analytic function near the origin, then its real zeros cannot accumulate at the origin. Indeed, if the solutions $(\varphi, z) = (\varphi^*, 0)$ were not isolated then $\Theta(\varphi, \rho_j(\varphi)) \equiv 0$ with $\rho_j(\varphi) = {\rm Re}(z_j(\varphi))$ meaning that $\Lambda_{pq} \neq \emptyset$. Of course it may happen that $(\varphi, z) = (\varphi^*, 0)$ be isolated but $\Lambda_{pq} \neq \emptyset$.
\newline

We consider the flow $\hat\Psi(t; z_0) = (\varphi(t; \rho_0), z(t; z_0))$ of \eqref{complex-edo} with initial condition $\hat\Psi(0; z_0) = (0, z_0)$ where $z_0 \in D \subset \mathbb{C}$ being $D$ a disc centered at the origin of radius sufficiently small and with $I \subset D$. Once we have proved that the singularities of \eqref{complex-edo} on $z=0$ are isolated, it follows that the saturation by the flow $\hat\Psi$ of the punctured disc $D_0 = D \setminus \{z=0\}$, that is $\hat\Psi(t; D_0)$, foliates the solid torus $\mathbb{T}$ excluding the line $z=0$ with
$$
\mathbb{T} \setminus \{z=0\} = \bigcup_{\stackrel{z_0 \in D_0}{\scriptscriptstyle 0 \leq t \leq T(z_0)}} \{  \hat\Psi(t; z_0) \},
$$
being $T(\rho_0) > 0$ the time such that $\varphi(T(\rho_0); \rho_0) = 2 \pi$.

We take the inverse integrating factor $V(\varphi, \rho)$ of the vector field $\mathcal{Z}$ and we know, see equation \eqref{charact-V}, that
$V(\varphi, \rho) = V(0, \rho_0) \, \exp\left( \int_0^t {\rm div}(\mathcal{Z}) \circ \Psi(\sigma; \rho_0) \, d\sigma \right)$, where we recall that $\Psi(t; \rho_0)$ is the flow of $\mathcal{Z}$ and $(\varphi, \rho) = \Psi(t; \rho_0)$ and that $V(0, \rho_0)$ is analytic in $I$. We consider its complexified version defining the function
$$
\hat{V}(\varphi, z) = V(0, z_0) \, \exp\left( \int_0^t {\rm div}(\mathcal{Z}) \circ \hat\Psi(\sigma; z_0) \, d\sigma \right)
$$
where $(\varphi, z) = \hat\Psi(t; z_0)$. For any $\varphi^* \in \Omega_{pq}$, we define the function $\hat{F}(z) = \hat{V}(\varphi^*, z)$ by
$$
\hat{V}(\varphi^*, z) = V(0, z_0) \, \exp\left( \int_0^{T^*(\rho_0)} {\rm div}(\mathcal{Z}) \circ \hat\Psi(\sigma; z_0) \, d\sigma \right)
$$
with $\hat\Psi(T^*(\rho_0); z_0) = (\varphi^*, z)$. Then $\hat{F}(z)$ is holomorphic in a  punctured disc of the origin, hence by Laurent theorem (see for example \cite{DyEd}) it admits a Laurent series expansion
$$
\hat{F}(z) = \sum_{n=-\infty}^{\infty} a_n z^n.
$$
We notice that, by construction, the real-valued restriction of $\hat{F}$ to $I_0$ is $V(\varphi^*, \rho)$, hence we obtain the real Laurent expansion
$$
V(\varphi^*, \rho) = \sum_{n=-\infty}^{\infty} a_n \rho^n,
$$
valid in $I_0$. Therefore the function $V(\varphi, \rho)$ has at $\rho=0$ the Laurent expansion
\begin{equation}\label{serie-laurent}
V(\varphi, \rho) = \sum_{n=-\infty}^{\infty} v_n(\varphi) \rho^n,
\end{equation}
where $v_n(\varphi^*) = a_n$. Notice that in particular we have proved that the coefficient functions $v_n$ are bounded and well defined in $\mathbb{S}^1$.
\end{proof}

\begin{remark}
{\rm The fact that a real function $f(\rho)$ be analytic at $I_0$ does not implies, in general, that it admits a Laurent series representation. A clear example is the piecewise function $f(\rho) = 1$ for $\rho > 0$ and $f(\rho) = 0$ for $\rho < 0$. Other interesting example is given by the complex function defined by a series $F(z) = \sum_{n \geq 1} 1/(1 + n^2 z^2)$ that has the sequence of singularities $\{ z_n = \pm i/n \}_{n \in \mathbb{N}}$ accumulating at the origin, hence $F(z)$ has not a Laurent series representation at $z=0$. The real function $f(\rho) = F(\rho)$ converges in $\mathbb{R} \setminus \{0\}$ just comparing with $\sum_{n \geq 1} 1/n^2$ . }
\end{remark}

\begin{remark}
{\rm We restrict ourselves to a particular inverse integrating factor $W(\varphi, \rho)$ of $\mathcal{Z}$ that exists for any center and comes from its associated inverse integrating factor $v(x,y)$ of $\mathcal{X}$. This $v(x,y)$ exists and is $C^\infty$ in a neighborhood of the center at $(x,y)=(0,0)$, see \cite{GP}, and analytic in a punctured neighborhood of the origin. We emphasize that the function $v$ can be flat at the origin. The relation between $W$ and $v$ is given by
\begin{equation}\label{V-polar-def}
W(\varphi, \rho) = \frac{v(\rho^p \cos\varphi, \rho^q \sin\varphi)}{\rho^{r} \, J(\varphi, \rho)}
\end{equation}
where $W$ is defined in $\hat{C} \backslash \{\rho=0\}$ and $J(\varphi, \rho) = \rho^{p+q-1}(p \cos^2\varphi + q \sin^2\varphi)$ is the Jacobian of the weighted polar blow-up (which only vanishes at $\rho=0$) and $r \in \mathbb{N}$ is the $(p,q)$--quasihomogeneous degree of the leading vector field associated to $\mathcal{X}$, see \cite{GaGi3}. Notice that by construction, in case that $v$ is not flat at the origin, $W$ is a Laurent inverse integrating factor $W(\varphi, \rho) = \sum_{j \geq m} w_j(\varphi) \rho^{j}$ with $w_m(\varphi) \not\equiv 0$. On the contrary, $W$ is flat at $\rho=0$ when $v$ is flat at the origin and, in this case $W(0, \rho) = v(\rho^p, 0) / \rho^{r+p+q-1})$ cannot be analytic in a neighborhood of $\rho=0$ except the trivial case $W(0, \rho) \equiv 0$.  }
\end{remark}

\begin{remark}
{\rm In Proposition 7 of \cite{GGR} there appears an example of a center without Laurent first integral $H(\varphi, \rho)$ with leading term of the polar vector field $\mathcal{Z}$, that is satisfying $\mathcal{Z}(H) = 0$ and possessing a Laurent expansion at $\rho=0$ with leading term. It is worth to emphasize that if one uses the techniques of the proof of Theorem \ref{center-V} adapted to try to prove that any center has a Laurent first integral with leading term we find an obstruction. It is that although the characteristic method allows to see in the center case that there is a first integral $H$ analytic in $\hat{C} \setminus \{ \rho=0 \}$ and even proving that the function $G_{\bar{\varphi}}(\rho) = H(\bar{\varphi}, \rho)$ for any $\bar{\varphi} \in \mathbb{S}^1$ has a Laurent series expansion $G_{\bar{\varphi}}(\rho) = \sum_{n \geq m} b_n \rho^n$ with $m \in \mathbb{Z}$, it may happen that this series has $m=0$ and $b_n = 0$ for $n > 0$, so that $H$ reduces to the trivial constant $H=b_0$. This phenomena is just what happens in the example of \cite{GGR}. We also mention that Lemma 3 of \cite{GGR} proves that any monodromic singularity of $\mathcal{X}$ such that its associated $\mathcal{Z}$ has a Puiseux first integral must be a center. }
\end{remark}

\subsection{Proof of Theorem \ref{center-Pi}}

We will use the same notation than that of the proof of Theorem \ref{center-V}.

Let the origin be a monodromic singularity of the analytic vector field $\mathcal{X}$ and $\Psi(t; \rho_0)$ be the (analytic) flow associated to the associated polar vector field $\mathcal{Z}$ with initial condition $\Psi(0; \rho_0) = (0, \rho_0) \in \hat{C} \setminus \{ \rho=0 \}$, that is for $0 < |\rho_0| \ll 1$. We define the Poincar\'e map $\Pi$ by $\Psi(T(\rho_0); \rho_0) = (2 \pi, \Pi(\rho_0))$ where $T$ is the flight return time function. We known that $T$ has a singularity at $\rho_0=0$ since  $\lim_{\rho_0 \to 0} T(\rho_0) = \infty$ when $\Omega_{pq} \neq \emptyset$. In that case also it is well known that the flow $\Psi(t; \rho_0)$ is not analytic at $\rho_0=0$.

Anyway we claim that $\Pi(\rho_0)$ is analytic in $I_0 = I \setminus \{0\}$ with $I \subset \mathbb{R}$ a small neighborhood of the origin when we restrict the parameters to $\Lambda \setminus \Lambda_{pq}$. The key point is that the differential equation
\begin{equation}\label{eq3**}
\frac{d \rho}{d \varphi} \, = \,  \frac{R(\varphi, \rho)}{\Theta(\varphi, \rho)},
\end{equation}
of the orbits of $\mathcal{Z}$ is analytic in $\hat{C} \backslash \Theta^{-1}(0)$ so that it is analytic in $\hat{C} \setminus \{ \rho=0 \}$ when we restrict the parameters to $\Lambda \setminus \Lambda_{pq}$. Then we consider the solution $\Phi(\varphi; \rho_0)$ of the Cauchy problem (\ref{eq3**}) with initial condition $\Phi(0; \rho_0) = \rho_0 \neq 0$ sufficiently small. It is clear by construction that $\Pi(\rho_0) = \Phi(2 \pi; \rho_0)$ so that $\Pi$ is analytic in $I_0$ proving the claim.

We consider again the complex flow $\hat\Psi(t; z_0) = (\varphi(t; \rho_0), z(t; z_0))$ of the real-complex system \eqref{complex-edo} with initial condition $\hat\Psi(0; z_0) = (0, z_0)$ where $z_0 \in D_0 \subset \mathbb{C}$ the small disc with the origin removed. Assuming its singularities $(\varphi^*, 0)$ are isolated, we define the complex Poincar\'e map $\hat{\Pi}(z) : D_0 \to \mathbb{C}$ as $\hat\Psi(T(\rho_0); z_0) = (2 \pi, \hat{\Pi}(z_0))$. Then $\hat{\Pi}$ is holomorphic at $D_0$, hence (invoking Laurent theorem) it admits a convergent Laurent series centered at the origin. Going back to the reals, since $\hat{\Pi}$ restricted to $I_0 \subset D_0$ coincides with $\Pi$, the real Poincar\'e map $\Pi$ also has a Laurent series convergent in $I_0$. Once we know that $\Pi$ has linear part at the origin, the proof finishes taking into account the structure $\Pi(\rho_0) = \eta_1 \rho_0 + o(\rho_0)$ proved in \cite{Ilyashenko}.

\subsection{Proof of Theorem \ref{teo-Pi-Algorithm}}

Following the proof of Theorem \ref{center-V}, and taking into account the arbitrariness of the initial condition $V(0, \rho)$, we take it analytic in a neighborhood $I \subset \mathbb{R}$ of the origin and not only
in $0 < \rho_0 \ll 1$ as in that proof. Notice that $\{ \varphi = 0 \}$ is a transversal to the flow of $\mathcal{Z}$ because $0 \not\in \Omega_{pq}$. Then we apply the characteristic method to obtain the solution $V(\varphi, \rho)$ of the Cauchy problem formed by the partial differential equation \eqref{edp-V} and the analytic initial condition $V(0, \rho)$ at $\rho=0$. It follows that $V$ can be uniquely analytically extended to $\rho=0$ except, perhaps, at the singularities $(\varphi, \rho) = (\varphi^*, 0)$ for all $\varphi^* \in \Omega_{pq}$ thanks to  Cauchy-Kovalevskaya Theorem.

From the unicity of coefficients in the expansions of the Laurent series \eqref{V-Laurent} it follows \eqref{V-Laurent} becomes a Taylor series when $\varphi \not\in \Omega_{pq}$ finishing the proof.
\newline

\noindent {\bf Acknowledgments.} The authors are very grateful to Professor H\'ector Giacomini for detecting mistakes in a draft version of this work that has led to significant changes in the corrected version.



\begin{thebibliography}{99}

\bibitem{ADGG} {\sc A. Algaba, M. D\' \i az, C. Garc\' \i a, J. Gin\'e}, {\it Center problem for generic degenerate vector fields}, Nonlinear Anal. {\bf 214} (2022), Paper No. 112597, 23 pp.

\bibitem{AGG} {\sc A. Algaba,  C. Garc\' \i a, J. Gin\'e}, {\it  Center conditions to find certain degenerate centers with characteristic directions}, Math. Comput. Simulation {\bf 215} (2024), 628--638.

\bibitem{AGR} {\sc A. Algaba, C. Garc\' \i a, M. Reyes}, {\it Characterization of a monodromic singular point of a planar vector field}, Nonlinear Anal. {\bf 74} (2011), 5402--5414.

\bibitem{AGR2} {\sc A. Algaba, C. Garc\' \i a, M. Reyes}, {\it A new algorithm for determining the monodromy of a planar differential system}, Appl. Math. Comput. {\bf 237} (2014), 419--429.

\bibitem{Ar} {\sc V.I. Arnold}, {\it Chapitres Suppl\'ementaires de la Th\'eorie des \'Equations Diff\'erentielles Ordinaires}, Moscow: Editions Mir. 1980.

 \bibitem{BerroneGiacomini}{\sc L.R. Berrone, H. Giacomini}, {\it On the vanishing set of inverse integrating factors}, Qual. Theory Dyn. Syst. {\bf 1} (2000), 211--230.

\bibitem{BRY}{\sc M. Briskin, N. Roytvarf, Y. Yosef}, {\it Center conditions at infinity for Abel differential equations},  Ann. of Math. (2) {\bf 172} (2010), no. 1, 437--483.

\bibitem{Bruno} {\sc A.D. Bruno}, Local methods in nonlinear differential equations. Springer-Verlag, Berlin, 1989.

\bibitem{CYZ} {\sc H. Chen, Y. Liu, X. Zeng}, { \it Center conditions and bifurcation of limit cycles at degenerate singular points in a quintic polynomial differential system}, Bull. Sci. Math. {\bf 129} (2005), no. 2, 127--138.

\bibitem{Dum} {\sc F. Dumortier,} {\it \ Singularities of vector fields on the plane}, J. Differential Equations
    {\bf 23} (1977), 53--106.

\bibitem{DyEd} {\sc R.H. Dyer, D.E. Edmunds}, From real to complex analysis. Springer Undergraduate Mathematics Series, 2014.

\bibitem{EncPer} {\sc A. Enciso, D. Peralta-Salas},{\it \ Existence and vanishing set of inverse integrating factors for analytic vector fields}, Bull. London Math. Soc. {\bf 41} (2009), 1112--1124.

\bibitem{GaGiGr} {\sc I.A. Garc\'{\i}a, H. Giacomini, M. Grau}, {\it The inverse integrating factor and the Poincar\'e map}, Trans. Amer. Math. Soc. {\bf 362} (2010), 3591--3612.

\bibitem{Ga-Gi} {\sc I.A. Garc\' \i a, J. Gin\'e},  {\it Center problem with characteristic directions and inverse integrating factors}, Commun. Nonlinear Sci. Numer. Simul. {\bf 108} (2022), 14 pp.

\bibitem{GaGi2}{\sc I.A. Garc\' \i a, J. Gin\'e}, {\it The Poincar\'e map of degenerate monodromic singularities with Puiseux inverse integrating factor}, Adv. Nonlinear Anal. {\bf 12} (2023), 20220314.

\bibitem{GaGi5}{\sc I.A. Garc\' \i a, J. Gin\'e}, {\it The linear term of the Poincar\'e map at singularities of planar vector fields}, J. Differential Equations, {\bf 396} (2024), 44--67.

\bibitem{GaGi3}{\sc I.A. Garc\' \i a, J. Gin\'e}, {\it Characterization of centers by its complex separatrices}, J. Differential Equations {\bf 42} (2025), 113506.

\bibitem{GaGi4}{\sc I.A. Garc\' \i a, J. Gin\'e}, {\it Principal Bautin ideal of monodromic singularities with inverse integrating factors}, J. Differential Equations {\bf 460} (2026), 114069.

\bibitem{GGGrau} {\sc I.A. Garc\' \i a, J. Gin\'e, M. Grau}, {\it A necessary condition in the monodromy problem for analytic differential equations on the plane}, J. Symbolic Comput. {\bf 41} (2006) 943--958.

\bibitem{GGR} {\sc I.A. Garc\' \i a, J. Gin\'e, A.L. Rodero}, {\it Existence and non-existence of Puiseux inverse integrating factors in analytic monodromic singularities}, Stud. Appl. Math. {\bf 153} (2024), e12724.

\bibitem{Ga-Ma} {\sc I.A. Garc\' \i a, S. Maza}, {\it A new approach to center conditions for simple analytic monodromic singularities}, J. Differential Equations {\bf 248} (2010), 363--380.

\bibitem{Ga-Ma-Ma} {\sc A. Gasull, V. Ma\~{n}osa, F. Ma\~{n}osas,} {\it \ Monodromy and stability of a class
    of degenerate planar critical points}, J. Differential Equations {\bf 217} (2005), 363--376.

\bibitem{G1} {\sc J. Gin\'e}, {\it On the degenerate center problem}, Internat. J. Bifur. Chaos Appl. Sci. Engrg. {\bf 21} (2011), no. 5, 1383--1392.

\bibitem{GM} {\sc J. Gin\'e, S. Maza}, {\it The reversibility and the center problem}, Nonlinear Anal {\bf 74} (2011), no. 2, 695--704.

\bibitem{GP} {\sc J. Gin\'e, D. Peralta-Salas}, {\it Existence of inverse integrating factors and Lie symmetries for degenerate planar centers}, J. Differential Equations {\bf 252} (2012), no. 1, 344--357.

\bibitem{GS} {\sc J. Gin\'e, D. Sinelshchikov}, {\it A new mechanism for producing degenerate centers in polynomial differential systems}, Nonlinear Anal. 264 (2026), Paper No. 113981, 10 pp.

\bibitem{Ilyashenko} {\sc Yu.S. Il'yashenko},{\it \ Finiteness theorems for limit cycles.} Translated from the Russian by H. H. McFaden. Translations of Mathematical Monographs, {\bf 94}. American Mathematical Society, Providence, RI, 1991.

\bibitem{Ilyashenko-2} {\sc Yu.S. Il'yashenko},{\it Algebraic unsolvability and almost algebraic solvability of the problem for the center-focus},  center-focus. Funkts. Anal. Prilozh. {\bf 6} (1972), 30--37.

\bibitem{Li-LLi-Xia} {\sc W. Li, J. Llibre, X. Zhang}, {\it \ Extension of Floquet's theory to nonlinear periodic differential systems and embedding diffeomorphisms in differential flows},  Amer. J. Math.  {\bf 124}  (2002), 107--127.

\bibitem{LWX} {\sc J. Llibre, Z. Wang, D. Xiao}, {\it Centers and Lyapunov quantities in a cubic polynomial Kolmogorov differential system}, Nonlinear Anal. Real World Appl. 89 (2026), Paper No. 104536, 14 pp.

\bibitem{Ma} {\sc V. Ma\~{n}osa}, {\it On the center problem for degenerate singular points of planar vector fields},
Internat. J. Bifur. Chaos Appl. Sci. Engrg. {\bf 12} (2002), no. 4, 687--707.

\bibitem{Me-2}{\sc N.B. Medvedeva}, {\it \ The principal term of the asymptotic expansion of the monodromy
    transformation: Calculation in blowing-up geometry}, Siberian Math. J. {\bf 38}, 114--126.

\bibitem{Me-3}{\sc N.B. Medvedeva, E. Batcheva}, {\it \ The second term of the asymptotics of the monodromy
    map in case of two even edges of Newton diagram}, Electron. J. Qual. Th. Diff. Eqs. {\bf 19}, 1--15.

\bibitem{Me}{\sc N.B. Medvedeva}, {\it On the analytic solvability of the problem of distinguishing between
    center and focus}, Proc. Steklov Inst. Math. {\bf 254} (2006), 7--93.

\bibitem{RS}{\sc V. G. Romanovski, D.S. Shafer}, The center and cyclicity problems: a computational
    algebra approach. Birkh\"auser Boston, Inc., Boston, MA, 2009.

 \bibitem{TLZ}{\sc Y. Tang, W. Li, Z. Zhang}, {\it Focus-center problem of planar degenerate system}, J. Math. Anal. Appl. {\bf 345} (2008), 934--940.

\end{thebibliography}
\end{document}